\documentclass[dvipdfmx]{amsart}
\usepackage{amsfonts}
\usepackage{amssymb}
\usepackage{amsmath}
\usepackage{amsthm}
\usepackage{amscd}
\usepackage{mathrsfs}
\usepackage{graphicx, color}
\usepackage{url}

\theoremstyle{theorem}
\newtheorem{theorem}{Theorem}[section]
\newtheorem{lemma}[theorem]{Lemma}
\newtheorem{corollary}[theorem]{Corollary}

\theoremstyle{definition}

\newcommand{\interior}{
 {\mathrm{int} \,}
}
\newcommand{\punc}{
 {\mathrm{punc} \,}
}

\begin{document}

\title[On the knot quandle of a fibered knot]{On the knot quandle of a fibered knot, \\ finiteness and equivalence of knot quandles}
\author{Ayumu Inoue}
\address{Department of Mathematics, Tsuda University, 2-1-1 Tsuda-machi, Kodaira-shi, Tokyo 187-8577, Japan}
\email{ayminoue@tsuda.ac.jp}


\begin{abstract}
We show that the structure of a fibered knot, as a fiber bundle, is reflected in its knot quandle.
As an application, we discuss finiteness and equivalence of knot quandles of concrete fibered 2-knots.
\end{abstract}

\maketitle

\section{Introduction}
\label{sec:introduction}

A quandle is an algebraic system, which has good chemistry with knot theory.
Associated with a knot, we have its knot quandle in a similar manner to the knot group.
Here and throughout this paper, a knot means the image of a smooth embedding of a closed, oriented, and connected $n$-dimensional manifold $M^{n}$ into the $(n+2)$-sphere $S^{n+2}$ with some $n \geq 1$.
We call it an $n$-knot if $M^{n}$ is homeomorphic to $S^{n}$.
In contrast with the similarity of their definitions, there are differences between behavior of knot quandles and knot groups.
It is known, by Joyce \cite{Joyce1982} and Matveev \cite{Matveev1982}, that every 1-knots are completely distinguished by their knot quandles up to equivalence, while knot groups do not.
Cardinalities of knot quandles are enabled to be finite, while those of knot groups are always infinite.
Indeed, the cardinality of the knot quandle of the trivial $n$-knot is obviously equal to 1 for any $n$.

An $n$-knot is said to be fibered if its complement has a `good' fiber bundle structure over $S^{1}$.
In this paper, we show that the knot quandle of a fibered knot is isomorphic to a quandle determined by the fundamental group of the fiber, its subgroup corresponding to the boundary, and the monodromy (Theorem \ref{thm:fibered}).
It asserts that the structure of a fibered knot, as a fiber bundle, might be reflected in objects obtained from its knot quandle, e.g., in quandle cocycle invariants \cite{CEGS2005}.

The $m$-twist-spin of an $n$-knot, introduced by Zeeman \cite{Zeeman1965}, is a typical method to obtain a fibered $(n+1)$-knot ($m \geq 1$).
As an application of Theorem \ref{thm:fibered}, we see that the cardinality of the knot quandle of the $m$-twist-spun trefoil is finite if (and only if) $1 \leq m \leq 5$ (Theorem \ref{thm:twist_spun_trefoil}).
We further see that the knot quandle of the $2$-twist-spun $2$-bridge knot of type $(p, q)$ is isomorphic to the dihedral quandle of order $p$ (Theorem \ref{thm:2bridge}).
It asserts that there are infinitely many $l$-tuples of mutually inequivalent $2$-knots having the same knot quandle for any $l \geq 2$ (Corollary \ref{cor:2bridge_2}).
We thus have an extension of the theorem, given by Tanaka \cite{Tanaka2007}, that there are infinitely many pairs of inequivalent 2-knots having the same knot quandle.

\section{Quandle}
\label{sec:quandle}

In this section, we review some notions about quandles briefly.
We refer the reader to \cite{CKS2004, EN2015, Joyce1982, Kamada2017} for more details.

A \emph{quandle} is a non-empty set $X$ equipped with a binary operation $\ast : X \times X \rightarrow X$ satisfying the following three axioms:
\begin{itemize}
\item[(Q1)]
For each $x \in X$, $x \ast x = x$
\item[(Q2)]
For each $x \in X$, a map $\ast \; x : X \rightarrow X$ ($w \mapsto w \ast x$) is bijective
\item[(Q3)]
For each $x, y, z \in X$, $(x \ast y) \ast z = (x \ast z) \ast (y \ast z)$
\end{itemize}
Notions of homomorphism and isomorphism are appropriately defined for quandles.

A typical example of a quandle is obtained from a group.
Let $G$ be a group and $\varphi$ an automorphism of $G$.
Then it is easy to see that the binary operation $\ast$ on $G$ defined by
\[
 x \ast y
 = \varphi(x y^{-1}) y
\]
satisfies the axioms of a quandle.
The quandle $(G, \ast)$ is usually called a \emph{generalized Alexander quandle} and denoted by $(G, \varphi)$.
Moreover, let $H$ be a subgroup of $G$ whose elements are fixed by $\varphi$.
Then the set of right cosets of $H$ in $G$, $H \backslash G$, is also a quandle with the binary operation $\ast$ given by
\[
 H x \ast H y
 = H \varphi(x y^{-1}) y.
\]
We refer to this quandle $(H \backslash G, \ast)$ as the \emph{quotient} of $(G, \varphi)$ by $H$.
Obviously, this quandle is isomorphic to $(G, \varphi)$ if $H$ is trivial.

We here introduce two basic quandles, both of which are generalized Alexander quandles on cyclic groups.
Suppose $G$ is a cyclic group $\mathbb{Z} / p \mathbb{Z}$ ($p \geq 1$), $\mathrm{id}$ the identity map of $G$, and $\mathrm{inv}$ the standard involution of $G$ mapping $a$ to $- a$.
The \emph{trivial quandle} of order $p$ is the generalized Alexander quandle $(G, \mathrm{id})$.
The \emph{dihedral quandle} of order $p$ is the generalized Alexander quandle $(G, \mathrm{inv})$.
Obviously, the trivial quandle of order $1$ and the dihedral quandle of order $1$ are the same quandle.
Therefore we only consider dihedral quandles whose orders are greater than or equal to $2$.

Associated with a knot $K$, we have its knot quandle as follows.
Suppose $N(K)$ is a tubular neighborhood of $K$ and $E(K) = S^{n+2} \setminus \interior{N(K)}$ the exterior.
Here, we assume that $N(K)$ is closed.
A (positive) \emph{meridional disk} $D$ of $K$ is an oriented disk properly embedded into $N(K)$ so that $K$ intersects with $D$ transversely only at one point with intersection number $+1$.
Choose and fix a point $p \in E(K)$.
A \emph{noose} of $K$ is a pair of a meridional disk $D$ and a path $\alpha$ in $E(K)$ from $\partial D$ to $p$.
The left-hand side of Figure \ref{fig:noose} depicts an image of a noose.
Two nooses $(D, \alpha)$ and $(E, \beta)$ are said to be \emph{homotopic} to each other if there is a pair of homotopies $H_{D} : D^{2} \times [0, 1] \rightarrow N(K)$ and $H_{P} : [0, 1] \times [0, 1] \rightarrow E(K)$ satisfying the following conditions:
\begin{itemize}
\item
For each $t \in [0, 1]$, $(H_{D}(D^{2} \times \{ t \}), H_{P}|_{[0, 1] \times \{ t \}})$ is a noose of $K$
\item
$(H_{D}(D^{2} \times \{ 0 \}), H_{P}|_{[0, 1] \times \{ 0 \}}) = (D, \alpha)$
\item
$(H_{D}(D^{2} \times \{ 1 \}), H_{P}|_{[0, 1] \times \{ 1 \}}) = (E, \beta)$
\end{itemize}
We denote the homotopy class of a noose $(D, \alpha)$ by $[(D, \alpha)]$.
Let $Q(K)$ be the set consisting of the homotopy classes of nooses of $K$.
We define a binary operation $\ast$ on $Q(K)$ by
\[
 [(D, \alpha)] \ast [(E, \beta)]
 = [(D, \alpha \cdot \beta^{-1} \cdot \partial E \cdot \beta)],
\]
where $\partial E$ denotes the loop starting at the initial point of $\beta$ and going round the boundary of $E$ once along the orientation.
The right-hand side of Figure \ref{fig:noose} illustrates what happens if we take this operation.
It is routine to check that $\ast$ satisfies the axioms of a quandle.
We refer to the quandle $(Q(K), \ast)$ as the \emph{knot quandle} of $K$ and simply denote it by $Q(K)$.
By definition, it is obvious that equivalent knots have the same knot quandle.
\begin{figure}[htbp]
 \begin{center}
  \includegraphics[scale=0.25]{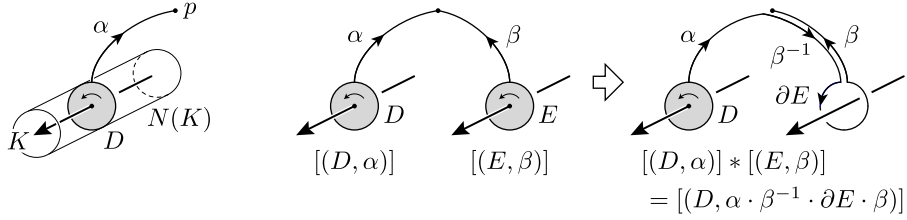}
 \end{center}
 \caption{A noose $(D, \alpha)$ (left) and the operation $\ast$ on $Q(K)$ (right)}
 \label{fig:noose}
\end{figure}

\section{The knot quandle of a fibered knot}
\label{sec:the_knot_quandle_of_a_fibered_knot}

An $n$-knot $K$ is said to be \emph{fibered} if there is a fibration $\pi : S^{n+2} \setminus K \rightarrow S^{1}$ such that each closure of $\pi^{-1}(\{ s \})$ is a Seifert surface of $K$ ($s \in S^{1}$).
Here, a Seifert surface of $K$ is an orientable $(n + 1)$-dimensional manifold, with boundary $S^{n}$, embedded into $S^{n+2}$ smoothly so that $K$ bounds it.
In this section, we see that the structure of a fibered knot, as a fiber bundle, is reflected in its knot quandle.
To state the claim correctly, we start with preparing some notions and notations.

Let $K$ be a fibered $n$-knot and $\pi : S^{n+2} \setminus K \rightarrow S^{1}$ its fibration with fiber $F$.
We refer to $F$ as the \emph{fiber} of $K$.
Regarding $S^{1}$ as the unit circle in the complex plane, we let $F_{0}$ be the preimage $\pi^{-1}(\{ 1 \})$.
Associated with the universal covering $\mathbb{R} \rightarrow S^{1}$ ($t \mapsto \exp \left( 2 t \pi \sqrt{-1} \right)$) of the base space of $\pi$, we have the infinite cyclic covering $\Psi : F \times \mathbb{R} \rightarrow S^{n+2} \setminus K$.
We note that, by definition, the preimage $\Psi^{-1}(F_{0})$ coincides with $F \times \mathbb{Z}$.

Suppose $N(K)$ is a tubular neighborhood of $K$ and $E(K) = S^{n+2} \setminus \interior{N(K)}$ the exterior.
Choose and fix a meridional disk $D_{0}$ of $K$ and a point $p_{0} \in \partial E(K)$ at which $D_{0}$ intersects with $F_{0}$ transversely.
We adopt $p_{0}$ as base points of the knot quandle $Q(K)$ of $K$ and the fundamental group $G$ of $F$ identifying $F$ and $F_{0}$.

For each loop $\gamma$ in $F_{0}$ based at $p_{0}$, since its lift to $F \times \mathbb{R}$ is a loop, the loop $(\partial D_{0})^{-1} \cdot \gamma \cdot \partial D_{0}$ is homotopic to a loop $\delta$ in $F_{0}$ relative to the base point $p_{0}$.
We thus have the automorphism $\varphi : G \rightarrow G$ sending $[\gamma]$ to $[\delta]$.
We refer to $\varphi$ as the \emph{monodromy} of $K$.

Let $H$ be the subgroup of $G$ consisting of the homotopy classes of loops in $F_{0} \cap N(K)$ based at $p_{0}$.
We note that this subgroup $H$ corresponds to $\partial F$, because the closure of $F_{0} \cap N(K)$ retracts to $\partial F_{0}$.
Since $F_{0} \cap N(K)$ retracts to $F_{0} \cap \partial N(K)$, each element of $H$ has a representative $\gamma$ in $F_{0} \cap \partial N(K)$.
Further, since $\partial N(K)$ is homeomorphic to $S^{n} \times S^{1}$, $(\partial D_{0})^{-1} \cdot \gamma \cdot \partial D_{0}$ is homotopic to $\gamma$ relative to $p_{0}$.
It asserts that the monodromy $\varphi$ fixes each element of $H$.
We are thus allowed to consider the quotient of the generalized Alexander quandle $(G, \varphi)$ by $H$.

We are now able to express our claim explicitly as follows.

\begin{theorem}
\label{thm:fibered}
Let $K$, $G$, $H$, and $\varphi$ be as above, i.e., $K$ a fibered $n$-knot, $F$ its fiber, $G$ the fundamental group of $F$, $H$ the subgroup of $G$ corresponding to $\partial F$, and $\varphi$ the monodromy of $K$.
Then the knot quandle $Q(K)$ of $K$ is isomorphic to the quotient of the generalized Alexander quandle $(G, \varphi)$ by $H$.
\end{theorem}

Since $H$ is trivial if $n \geq 2$, we have the following corollary immediately.

\begin{corollary}
\label{cor:fibered}
Let $K$, $G$, and $\varphi$ be as above, i.e., $K$ a fibered $n$-knot, $F$ its fiber, $G$ the fundamental group of $F$, and $\varphi$ the monodromy of $K$.
Then the knot quandle $Q(K)$ of $K$ is isomorphic to the generalized Alexander quandle $(G, \varphi)$ if $n$ is greater than or equal to $2$.
\end{corollary}

To prove the theorem, we first show the following two lemmas.

\begin{lemma}
\label{lem:fibered_1}
For each noose $(D, \alpha)$ of $K$, there is a loop $\beta$ in $F_{0} \cap E(K)$ based at $p_{0}$ such that the noose $(D_{0}, \beta)$ is homotopic to $(D, \alpha)$.
\end{lemma}

\begin{proof}
Since $\partial E(K)$ is path-connected, we have a path $\gamma$ in $\partial E(K)$ from $p_{0}$ to the initial point of $\alpha$.
As illustrated in Figure \ref{fig:moving}, the noose $(D, \alpha)$ is homotopic to the noose $(D_{0}, \gamma \cdot \alpha)$.
\begin{figure}[htbp]
 \begin{center}
  \includegraphics[scale=0.25]{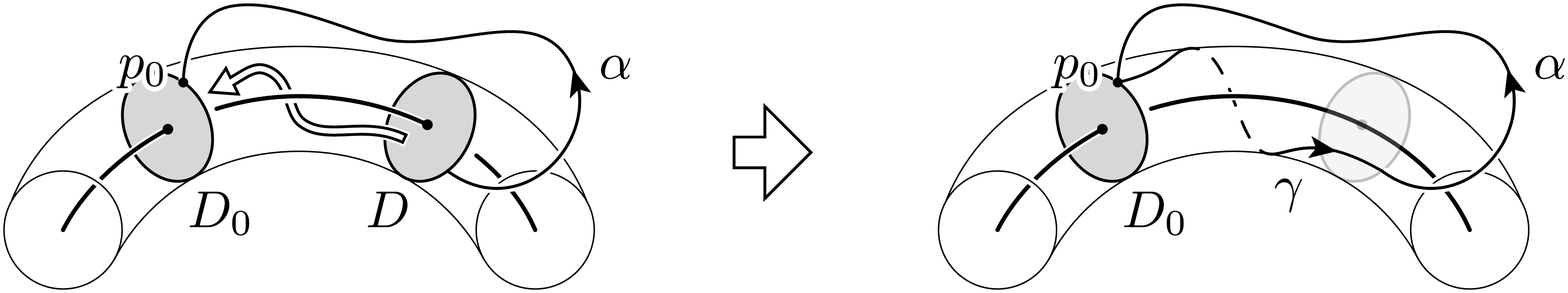}
 \end{center}
 \caption{Sliding $D$ to $D_{0}$ along $\gamma^{-1}$, we obtain $(D_{0}, \gamma \cdot \alpha)$ from $(D, \alpha)$}
 \label{fig:moving}
\end{figure}

Suppose $\widetilde{\gamma \cdot \alpha}$ is the lift of $\gamma \cdot \alpha$ to $F \times \mathbb{R}$ starting at $\Psi^{-1}(\{ p_{0} \}) \cap (F \times \{ 0 \})$.
We note that the terminal point of $\widetilde{\gamma \cdot \alpha}$ is then $\Psi^{-1}(\{ p_{0} \}) \cap (F \times \{ i \})$ with some $i \in \mathbb{Z}$.
Let $\widetilde{(\partial D_{0})^{i}}$ be the lift of $(\partial D_{0})^{i}$ to $F \times \mathbb{R}$ starting at $\Psi^{-1}(\{ p_{0} \}) \cap (F \times \{ 0 \})$ and thus ending at $\Psi^{-1}(\{ p_{0} \}) \cap (F \times \{ i \})$.
Then, since $F \times \mathbb{R}$ retracts to $F$, there is a loop $\widetilde{\beta}$ in $\Psi^{-1}(F_{0} \cap E(K)) \cap (F \times \{ i \})$ such that $\widetilde{(\partial D_{0})^{i}} \cdot \widetilde{\beta}$ is homotopic to $\widetilde{\gamma \cdot \alpha}$ relative to the end points.
We thus have the noose $(D_{0}, (\Psi \circ \widetilde{(\partial D_{0})^{i}}) \cdot (\Psi \circ \widetilde{\beta})) = (D_{0}, (\partial D_{0})^{i} \cdot (\Psi \circ \widetilde{\beta}))$ which is homotopic to $(D_{0}, \gamma \cdot \alpha)$.

As illustrated in Figure \ref{fig:winding}, the noose $(D_{0}, (\partial D_{0})^{i} \cdot (\Psi \circ \widetilde{\beta}))$ is homotopic to the noose $(D_{0}, \Psi \circ \widetilde{\beta})$.
Since the loop $\Psi \circ \widetilde{\beta}$ lies in $F_{0} \cap E(K)$, our claim is proved letting $\beta$ denote $\Psi \circ \widetilde{\beta}$.
\begin{figure}[htbp]
 \begin{center}
  \includegraphics[scale=0.25]{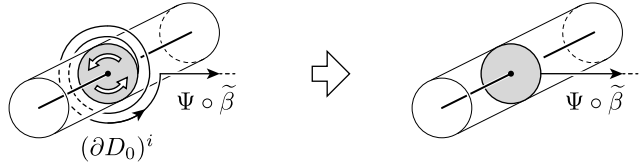}
 \end{center}
 \caption{Rotating $D_{0}$ appropriately, we obtain $(D_{0}, \Psi \circ \widetilde{\beta})$ from $(D_{0}, (\partial D_{0})^{i} \cdot (\Psi \circ \widetilde{\beta}))$}
 \label{fig:winding}
\end{figure}
\end{proof}

\begin{lemma}
\label{lem:fibered_2}
Let $\alpha$ and $\beta$ be loops in $F_{0} \cap E(K)$ based at $p_{0}$.
Then the noose $(D_{0}, \alpha)$ is homotopic to the noose $(D_{0}, \beta)$ if and only if $[\alpha]$ is an element of the right coset $H [\beta]$.
\end{lemma}

\begin{proof}
We first assume that $[\alpha]$ is an element of $H [\beta]$.
Then, there is a loop $\gamma$ in $F_{0} \cap \partial E(K)$ such that $\gamma \cdot \beta$ is homotopic to $\alpha$.
Since the noose $(D_{0}, \gamma \cdot \beta)$ is homotopic to both of $(D_{0}, \alpha)$ and $(D_{0}, \beta)$, $(D_{0}, \alpha)$ is homotopic to $(D_{0}, \beta)$.

Conversely, we next assume that $(D_{0}, \alpha)$ is homotopic to $(D_{0}, \beta)$ by a pair of homotopies $H_{D} : D^{2} \times [0, 1] \rightarrow N(K)$ and $H_{P} : [0, 1] \times [0, 1] \rightarrow E(K)$ satisfying $(H_{D}(D^{2} \times \{ 0 \}), H_{P}|_{[0, 1] \times \{ 0 \}}) = (D, \alpha)$ and $(H_{D}(D^{2} \times \{ 1 \}), H_{P}|_{[0, 1] \times \{ 1 \}}) = (E, \beta)$.
We define homotopies $H_{D}^{\prime} : D^{2} \times [0, 1] \rightarrow N(K)$ and $H_{P}^{\prime} : [0, 1] \times [0, 1] \rightarrow E(K)$ by
\[
 H_{D}^{\prime}(x, t) =
 \begin{cases}
  H_{D}(x, 0) & \enskip \left( 0 \leq t \leq \dfrac{1}{2} \right), \\[2ex]
  H_{D}(x, 2 t - 1) & \enskip \left( \dfrac{1}{2} \leq t \leq 1 \right)
 \end{cases}
\]
and
\[
 H_{P}^{\prime}(s, t) =
 \begin{cases}
  H_{P}(0, 2 s) & \enskip \left( 0 \leq s \leq t, \, 0 \leq t \leq \dfrac{1}{2} \right), \\[2ex]
  H_{P} \left( \dfrac{s - t}{1 - t}, 2 t \right) & \enskip \left( t \leq s \leq 1, \, 0 \leq t \leq \dfrac{1}{2} \right), \\[2ex]
  H_{P}(0, 2 s + 2 t - 1) & \enskip \left( 0 \leq s \leq 1 - t, \, \dfrac{1}{2} \leq t \leq 1 \right), \\[2ex]
  H_{P} \left( \dfrac{s + t - 1}{t}, 1 \right) & \enskip \left( 1 - t \leq s \leq 1, \, \dfrac{1}{2} \leq t \leq 1 \right).
 \end{cases}
\]
Then, obviously, the pair of homotopies $H_{D}^{\prime}$ and $H_{P}^{\prime}$ is also a homotopy between $(D_{0}, \alpha)$ and $(D_{0}, \beta)$.
This homotopy asserts that the loop $\alpha \cdot \beta^{-1}$ is homotopic to the loop $\gamma = H_{P}^{\prime}|_{[0, 1 / 2] \times \{ 1 / 2 \}}$ in $\partial E(K)$ relative to the base point $p_{0}$.
Since the loop $\alpha \cdot \beta^{-1}$ lies in $F_{0}$, a lift of $\gamma$ to $F \times \mathbb{R}$ is a loop.
Thus $\gamma$ is homotopic to a loop $\gamma^{\prime}$ in $F_{0} \cap \partial E(K)$ relative to $p_{0}$.
In conclusion, we found that $\alpha$ is homotopic to $\gamma^{\prime} \cdot \beta$ relative to $p_{0}$.
Since $[\gamma^{\prime}]$ is an element of $H$, $[\alpha]$ is an element of $H [\beta]$.
\end{proof}

Now we prove the theorem.

\begin{proof}[Proof of Theorem \ref{thm:fibered}]
Lemmas \ref{lem:fibered_1} and \ref{lem:fibered_2} assert that the map $f : H \backslash G \rightarrow Q(K)$ sending $H [\alpha]$ to $[(D_{0}, \alpha)]$ is well-defined and bijective.
Further $f$ preserves the quandle operation, because we have the following equations:
\begin{eqnarray*}
 [(D_{0}, \alpha)] \ast [(D_{0}, \beta)]
 & = & [(D_{0}, \alpha \cdot \beta^{-1} \cdot \partial D_{0} \cdot \beta)] \\
 & = & [(D_{0}, (\partial D_{0})^{-1} \cdot \alpha \cdot \beta^{-1} \cdot \partial D_{0} \cdot \beta)] \\
 & = & [(D_{0}, \varphi(\alpha \cdot \beta^{-1}) \cdot \beta)],
\end{eqnarray*}
where $\varphi(\alpha \cdot \beta^{-1})$ denotes a representative of $\varphi([\alpha \cdot \beta^{-1}])$.
We remark that the second equality follows from a similar argument depicted in Figure \ref{fig:winding}.
Thus $f$ is an isomorphism from the quotient of the generalized Alexander quandle $(G, \varphi)$ by $H$ to the knot quandle $Q(K)$.
\end{proof}

\section{Finiteness and equivalence of knot quandles}
\label{sec:finiteness_and_equivalence_of_knot_quandles}

Associated with a 1-knot $k$ and a positive integer $m$, as mentioned in Section \ref{sec:introduction}, we have a fibered 2-knot called the \emph{$m$-twist-spun $k$}.
We denote the $m$-twist-spun $k$ by $\tau^{m} k$.
In this section, in light of Corollary \ref{cor:fibered}, we investigate knot quandles of several $\tau^{m} k$ and discuss their finiteness and equivalence.

Let $\Sigma_{m}$ denote the $m$-fold cyclic branched covering of $S^{3}$ along $k$ and $\punc{\Sigma_{m}}$ the once-punctured $\Sigma_{m}$.
It is known, by Zeeman \cite{Zeeman1965}, that the fiber of $\tau^{m} k$ is $\punc{\Sigma_{m}}$ and the monodromy $\varphi$ of $\tau^{m} k$ is induced by the covering transformation of $\Sigma_{m}$, which fixes $k$ and is of order $m$, in a natural way.
We only use the facts in the remaining.
We thus do not explain about twist-spinning any longer in this paper.
We refer the reader to \cite{CKS2004, Kamada2017, Zeeman1965} for more details on twist-spinning.

We first consider the case that $k$ is the trefoil knot.
In this case, it is computed in \cite[Subsection 10.D]{Rolfsen1976} that the fundamental group $\pi_{1}(\punc{\Sigma_{m}})$ is isomorphic to the cyclic group $\mathbb{Z} / 3 \mathbb{Z}$, the quaternion group, the binary tetrahedral group, or the binary icosahedral group if $m$ is equal to $2$, $3$, $4$, or $5$ respectively.
We note that the cardinalities of those groups are $3$, $8$, $24$, and $120$ respectively.
Further $\pi_{1}(\punc{\Sigma_{1}})$ is trivial, because $\Sigma_{1}$ is homeomorphic to $S^{3}$.
We thus have the following theorem immediately.

\begin{theorem}
\label{thm:twist_spun_trefoil}
The cardinality of the knot quandle of the $m$-twist-spun trefoil is $1$, $3$, $8$, $24$, or $120$ if $m$ is equal to $1$, $2$, $3$, $4$, or $5$ respectively.
\end{theorem}

The knot quandle of the $1$-twist-spun trefoil is obviously isomorphic to the trivial quandle of order 1.
We will see later that the knot quandle of the $2$-twist-spun trefoil is isomorphic to the dihedral quandle of order 3.
It will be discussed in \cite{InoueIP} that the knot quandle of the $3$-, $4$-, or $5$-twist-spun trefoil is isomorphic to a quandle derived from rotational symmetries of the 16-, 24-, or 600-cell respectively.

Moreover it will be discussed in \cite{InoueIP} that the cardinality of the knot quandle of the $m$-twist-spun trefoil is infinite if $m$ is greater than or equal to $6$.
Thus the cardinality of the knot quandle of the $m$-twist-spun trefoil is finite if and only if $1 \leq m \leq 5$.

We turn to the next case.
Suppose $p$ is a positive odd integer and $q$ an integer satisfying $(p, q) = 1$ and $|q| < p$.
Associated with $p$ and $q$, we have a 1-knot called the \emph{$2$-bridge knot} of type $(p, q)$ (see \cite{BZH2014} for example).
We let $k$ be the $2$-bridge knot of type $(p, q)$.
In this case, it is known that $\Sigma_{2}$ is homeomorphic to the lens space $L(p, q)$.
Thus the fundamental group $\pi_{1}(\punc{\Sigma_{2}})$ is isomorphic to the cyclic group $\mathbb{Z} / p \mathbb{Z}$.
Further the monodromy $\varphi$ of $\tau^{2} k$ is the standard involution of $\mathbb{Z} / p \mathbb{Z}$.
Indeed, consider a Heegaard splitting of $\Sigma_{2}$ into two solid tori.
Then each restriction of the covering transformation of $\Sigma_{2}$ to a solid torus is the extension of the hyperelliptic involution of its boundary (see \cite{HR1985} for example).
We thus have the following theorem.

\begin{theorem}
\label{thm:2bridge}
The knot quandle of the $2$-twist-spun $2$-bridge knot of type $(p, q)$ is isomorphic to the dihedral quandle of order $p$.
\end{theorem}

We note that the trefoil knot is equivalent to the $2$-bridge knot of type $(3, 1)$.
Therefore the knot quandle of the $2$-twist-spun trefoil is isomorphic to the dihedral quandle of order $3$ as mentioned above.

In light of the works by Atiyah-Bott \cite[Remark for Theorem 7.27]{AB1968} and Brody \cite{Brody1960}, it is known that the $2$-twist-spun $2$-bridge knot of type $(p, q)$ and the $2$-twist-spun $2$-bridge knot of type $(p^{\prime}, q^{\prime})$ are equivalent to each other if and only if $p^{\prime} = p$ and $q^{\prime} \equiv \pm q^{\pm 1} \pmod{p}$.
We thus have the following corollary immediately.

\begin{corollary}
\label{cor:2bridge_1}
Let $p$ be a positive odd integer and $q_{i}$ an integer satisfying $(p, q_{i}) = 1$ and $|q_{i}| < p$ {\upshape (}$i = 1, 2${\upshape )}.
Suppose $k_{i}$ is the $2$-bridge knot of type $(p, q_{i})$.
Assume that $q_{2}$ is not equivalent to $\pm q_{1}^{\pm 1}$ modulo $p$.
Then, although they are inequivalent $2$-knots, the $2$-twist-spun $k_{1}$ and the $2$-twist-spun $k_{2}$ have the same knot quandle.
\end{corollary}

Suppose $l$ is an integer greater than or equal to $2$.
It is easy to see that, for sufficient large odd integers $p$, there are integers $q_{1}, q_{2}, \dots, q_{l}$ satisfying $(p, q_{i}) = 1$ ($1 \leq i \leq l$), $|q_{i}| < p$ ($1 \leq i \leq l$), and $q_{j} \not\equiv \pm q_{i}^{\pm 1} \pmod{p}$ ($1 \leq i < j \leq l$).
We thus have the following corollary in light of Theorem \ref{thm:2bridge} and Corollary \ref{cor:2bridge_1}.

\begin{corollary}
\label{cor:2bridge_2}
For any integer $l$ greater than or equal to $2$, there are infinitely many $l$-tuples of mutually inequivalent $2$-knots having the same knot quandle.
\end{corollary}

This corollary extends the theorem, given by Tanaka \cite{Tanaka2007}, that there are infinitely many pairs of inequivalent 2-knots having the same knot quandle.

\section*{Acknowledgments}
The author would like to express his thanks to Scott Carter, Katsumi Ishikawa, Akio Kawauchi, and Daniel Silver for useful conversations.
He is partially supported by JSPS KAKENHI Grant Number JP16K17591.

\bibliographystyle{amsplain}

\end{document}